\documentclass[10pt,psmasfonts]{article}
\usepackage{color,amsfonts, amsmath, amssymb,latexsym,amsthm}
\usepackage[margin=1.1in]{geometry}
\usepackage[all]{xy}

\def\red#1{\textcolor[rgb]{1,0.00,0.00}{#1}}

\def\???{\red{???}}

\newcounter{todoc}
\newtheorem{theorem}{Theorem}[section]
\newtheorem{lemma}[theorem]{Lemma}

\newtheorem{proposition}[theorem]{Proposition}
\newtheorem{remark0}[theorem]{Remark}
\newtheorem{example0}[theorem]{Example}

\newtheorem{question}[theorem]{Question}

\newenvironment{example}{\begin{example0}\rm}{\end{example0}}

\newcommand{\propref}[1]{Proposition~\ref{#1}}
\newcommand{\thmref}[1]{Theorem~\ref{#1}}
\newcommand{\lemref}[1]{Lemma~\ref{#1}}

\newcommand{\exref}[1]{Example~\ref{#1}}

\def\res{{\rm K}}

\def\series#1{{\res[\![x_1,\dots x_{#1}]\!]}}
\def\sers2{{\res[\![x,y]\!]}}
\def\ser3{{\res[\![x,y,z]\!]}}
\def\sery{{\res[\![y]\!]}}
\def\serz{{\res[\![z]\!]}}
\def\seryz{{\res[\![y,z]\!]}}

\def\polis#1{{\res[x_1,\dots x_{#1}]}}
\def\pol2{{\res[x,y]}}
\def\pol3{{\res[x,y,z]}}

\def\TT{{\mathbb {T}^n}}
\def\tauh{{\overline{\tau}}}
\def\lt#1{{\rm{Lt}}_{\tauh}(#1)}

\def\supp{{\rm{Supp}}}

\def\deg{{\rm{deg}}}
\def\order{{\rm{order}}}
\def\NF{{\rm NF}}

\def\ui{{\underline i}}

\def\dim{\operatorname{dim}}
\def\depth{\operatorname{depth}}

\newcommand{\m}{\mathfrak m}

\newcommand{\M}{\mathcal M}

\def\max{\operatorname{max}}

\title{  \bf   On the Hilbert function of
one-dimensional local complete intersections
}
\author{\large   J. Elias
\thanks{Partially supported by  MTM2010-20279-C02-01}
\and \large M. E. Rossi
\and \large G. Valla
 }

\begin{document}
\maketitle

\begin{abstract}
The Hilbert function of standard graded algebras are well understood by
Macaulay's theorem and very little is known in the local case,
even if we assume that the local ring is a complete intersection.
An extension to the power series ring $R$ of the theory of Gr\"{o}bner bases (w.r.t.
local degree orderings) enable us to 
characterize the Hilbert function  of  one dimensional quadratic complete intersections
 $A=R/I$, and we give a structure theorem of the minimal system of generators of $I$ in terms of the  Hilbert function.
We  find several restrictions for the Hilbert function of   $A$ in the case that  $I$ is a complete intersection of  type $(2,b). $  Conditions for the  Cohen-Macaulyness of the associated graded ring of $A$ are given.
\end{abstract}

\section{Introduction and preliminaries}

Let $G$ be a standard graded K-algebra; by this we mean $G=P/I$ where $P=\polis{n}$ is a polynomial ring over the field K and $I$ an homogeneous ideal. It is clear that  for every $t\ge 0$  the set $I_t$ of the forms of degree $t$ in $P$ is a K-vector space of finite dimension. For every positive integer $t$ the Hilbert function of $G$ is defined as follows: $$HF_G(t)=  dim_\res G_t= dim_\res P_t- dim_\res I_t=\binom{n+t-1}{t}- \ dim_\res I_t.$$ Its generating function $HS_G(\theta)=\sum_{t\in \mathbb{N}}HF_G(t) \theta^t$ is the Hilbert Series of $G.$

 The relevance of this notion comes from the fact  that in the case $I$ is the defining ideal of a projective variety $V,$ the dimension, the degree and the arithmetic genus  of $V$ can be immediately computed from the Hilbert Series of $P/I.$

 A fundamental theorem by Macaulay describes exactly those numerical functions which occur as the Hilbert functions  of a standard graded K-algebra. Macaulay's Theorem says that for each $t$ there is an upper bound for $HF_G(t+1)$ in terms of $HF_G(t)$, and this bound is sharp in the sense that any numerical function satisfying it can be realized as the Hilbert function of a suitable homogeneous standard K-algebra. These numerical functions are called ``admissible''  and will be described  in the next section.

 It is not surprising that additional properties yield further constraints on the Hilbert function.
 Thus, for example, the Hilbert function  of a Cohen-Macaulay  standard graded algebra is  completely described by another  theorem of Macaulay which says that the Hilbert series admissible for a Cohen-Macaulay standard graded algebra of dimension d, are of the type
  $$\frac{1+h_1\theta+\dots h_s\theta ^s}{(1-z)^d}$$  where $1+h_1\theta+\dots h_s\theta ^s$ is admissible.

\smallskip
The Hilbert function of a local ring $A$ with maximal ideal $\m$ and residue field K is defined as follows: for every $t \ge 0$
$$
HF_A(t) =   dim_\res \left(\frac{\m^t}{\m^{t+1}}\right).
$$

It is clear that $HF_A(t)$ is equal to the minimal number of generators of the ideal $\m^t$ and we can  see that the Hilbert function of the local ring $A$ is the Hilbert function of the following  standard graded algebra $$gr_{\m}(A)=\oplus_{t\ge 0}\ \m^t/ \m^{t+1}.$$ This algebra is called the associated graded ring of the local ring $(A,\m)$  and corresponds to a relevant geometric construction in the case $A$ is the localization at the origin O of the coordinate ring of an affine variety $V$ passing through O. It turns out that  $gr_{\m}(A)$ is the coordinate ring of the {\it Tangent Cone } of $V$ at O, which is the cone composed of all lines that are the limiting positions of secant lines to $V$ in O.

Despite the fact that the Hilbert function of a standard graded K-algebra $G$ is so well understood in the case
$G$ is Cohen-Macaulay, very little is known in the local case. This mainly because, passing from the local ring $A$ to its   associated  graded ring, many of the properties can be  lost.
This is the reason why we are very far from a description of the admissible  Hilbert functions for a Cohen-Macaulay local ring when $gr_{\m}(A)$ is not Cohen-Macaulay.
  We only have some small knowledge of the behavior of these numerical functions.

An example by Herzog and Waldi (see \cite{HW75}) shows that the Hilbert function of a one dimensional Cohen-Macaulay local ring can be decreasing, even  the number of generators of the square of the maximal ideal can be less than the number of generators of the maximal ideal itself.
Further, without restrictions on the embedding dimension, the Hilbert function of a one dimensional Cohen-Macaulay local ring can  present arbitrarily many "valleys"  (see \cite{eli94a}).

Even  if we restrict ourselves  to the case of a complete intersection, very little is known. In \cite{put05} it has been proved that the Hilbert function of a positive dimensional codimension two complete intersection $R/(f,g)$ is non decreasing, but we have no answer to the  question asked by Rossi (see \cite{Ros11}) whether the same is true for every one dimensional Gorenstein local ring.

    In the case that the embedding dimension of the local ring  is at most three, the first author gave a positive answer to a question stated by J. Sally, by proving that the Hilbert function of a one dimensional Cohen-Macaulay local ring is increasing (see \cite{Eli93a}).
 But  examples show that this is not true anymore if the embedding dimension is bigger than three.

 All this amount of results shows that, without strong assumptions, the Hilbert function of a one-dimensional Cohen-Macaulay local ring could be  very wild. This is the reason why, in this paper, we restrict ourselves to the case  $A=\ser3/I, $ where the ideal $I \subseteq (x,y,z)^2 $ is generated by a regular sequence $\{f,g\}$ of elements of $R$.
We will see that, even with all these  strong assumptions, the problem of determining  the admissible Hilbert functions is  not so easy, possibly   because it is strictly related to the study of curve  singularities in ${\mathbb {A}}^3.$

If we consider the corresponding Artinian problem, then we deal with a pair of plane curves. Several papers have been written in which the Hilbert function of an  Artinian  complete intersection ring $A=\sers2/(f,g) $ has been studied in terms of the invariants of the curves $f=g=0$ (see Iarrobino \cite{Iar94}, Goto, Heinzer, Kim   \cite{GHK08}, Kothari \cite{Kor78}, ....).      It is an early  result due to Macaulay that the Hilbert function of such a ring $A$ verifies for every positive integer $n$ the following inequality:
$$| HF_A(n+1) - HF_A(n)|  \le 1. $$
It has been proved  that given such a  numerical function,
there exists a complete intersection $I=(f,g) \subseteq \sers2 $ with that   Hilbert function, \cite{Ber09}, \cite{GHK08}.
Hence the problem is solved in the Artinian case and, more in general, when  $gr_{\m}(A)$ is Cohen-Macaulay.
Conditions on the Cohen-Macaulayness of  $gr_{\m}(A)$ have been studied by Goto, Heinzer and Kim  in \cite{GHK06}, \cite{GHK07}.

Classical results concerning Cohen-Macaulay local rings  of dimension one  will be useful in this paper. For example it is well known, see \cite{Mat77},\cite{Eli93a}, \cite{RV-Book},  that there exists an integer  $e \ge 1$, the multiplicity of $A, $  such that
\begin{enumerate}
\item[(i)]
$HF_A(n) \le e$ for all $n$,
\item[(ii)]
If $HF_A(j)=e$ for some  $j$, then $HF_A(n) = e$ for all $n\ge j$,
\item[(iii)]
For every $j\ge 0$ we have $HF_A(j) \ge \min \{j+1, e \}.$ In particular $HF_A(e-1)= e.$
\end{enumerate}

\noindent
The least integer $r $ such that $HF(r)=e $  coincides with  the reduction number of $\m, $ which is
the least integer $r$ such that $\m^{r+1}= x \m^r $ for some (hence any) superficial element $x \in \m.  $
We say that the Hilbert function of $A$  is  {\it {increasing}} (resp. {\it {strictly increasing}} )  if $HF(n) \le HF_A(n+1)$ (resp. $HF(n) < HF_A(n+1)$) for all $n=0,\cdots, r -1.$

Throughout the whole paper  $\res$ denotes an algebraically closed field
of characteristic zero.
Let $R=\series{n}$ be the ring of  formal power series in the indeterminates $\{x_1,\cdots,x_n\}$ with coefficients in $\res$  and  maximal ideal $\mathcal M =(x_1,\cdots,x_n).$
 We denote by  $\mathbb U (R)$ the group of units  of $R$.
Let  $ I$ be an ideal of $R$ and consider the local ring $A=R/I$ whose maximal ideal is   $\m:= {\mathcal M}/I. $

We have seen that the Hilbert function of a local ring $A$ is the same as that of the associated graded ring $gr_{\m}(A).$ Hence it will be useful to recall the presentation of this standard graded algebra.
 For every power series  $f\in R\setminus \{0\}$ we can write $f=f_v+f_{v+1}+\cdots$, where $f_v$ is not zero and $f_j$ is an homogeneous polynomial of degree $j$ in $P$ for every $j\ge v.$
 We say that $v$ is the order of $f$, denote $f_v$ by $f^*$ and call it the initial form of $f.$
 If $f=0$ we agree that its order is $\infty.$
  It is well known that $gr_{\m}(A)=P/I^*$,  where $I^*$ is the homogeneous ideal of the polynomial ring $P$ generated by the initial forms of the elements of $I.$
  A set of power series $f_1,\cdots,f_r\in I$ is a standard basis of $I$ if $I^*=(f_1^*,\cdots,f_r^*)$, (see \cite{Hir64a}).
  It is clear that every ideal $I$ has a standard basis and that every standard basis is a basis.
  However not every basis is a standard basis.
 To determine a standard basis of a given ideal of $R$ is a classical  hard problem, even in the very special case we are involved with in this paper.

In order to determine  the Hilbert function of such local complete intersections it seems to be  hopeless to use only the  theory of tangent cones. Instead  we found crucial to consider  the extension  to the power series ring of the theory of Gr\"obner bases introduced by  Buchberger  for  ideals in the polynomial ring. We can say that a mixture of the theory of enhanced standard basis with that of the ideals of initial forms has been the  winning strategy for us.
The use of the theory of enhanced standard bases for studying  of the Hilbert function of a local ring seems to be  unusual, while there are several papers in  the Theory of Singularities where this topic is essential.

We recall that the notion of Gr\"obner basis is defined by considering  a term ordering on
the terms of $P $ (i.e. a monomial  ordering where all the  terms are bigger than $1$).
Instead, we need here to consider the so called {\it {local degree ordering,  }} see \cite{singular}, Chapter 6, a monomial  ordering on
the terms of $P $ which is not a term ordering.

 \smallskip
We denote by $\TT$ the set of terms or monomials of $P$; let $\tau$ be a term ordering in $\TT$,
and we assume that $x_1>\cdots > x_n$.
We define a new total order $\tauh$ on $\TT$ in the following way:
given $m_1, m_2\in \TT$
we let
$m_1 >_{\tauh} m_2$ if and only if $\deg(m_1) < \deg(m_2)$ or $\deg(m_1) = \deg(m_2)$ and $m_1 >_{\tau} m_2$.
\smallskip
Given $f\in R$ we denote by $\supp(f)$ the support of $f$, i.e. if $f=\sum_{\ui\in \mathbb N^n} a_{\ui} x^{\ui}$
then $\supp(f)$ is the set of terms  $x^\ui$ such that $a_{\ui}\neq 0$. We remark that, given $f$ in $R$,  there is a monomial which is the biggest of the monomials in
$\supp(f)$ with respect to $\tauh$: namely, since the support of $f^*$ is a finite set, we can take  the maximum with respect to $\tau$ of the elements of this set. This monomial is called the leading monomial of $f$ with respect to $\tauh$ and is denoted by $Lt_{\tauh}(f).$ By definition we  have   $$Lt_{\tauh}(f)=Lt_{\tau}(f^*).$$

As usual we define the leading term ideal associated to an ideal $I\subset R$ as the   monomial ideal $\lt{I}$  generated in $R$ by $\lt{f}$ with $f$ running in $ I$.

\smallskip
In \cite{Ber09} a  set $\{ f_1, \dots , f_r\} $ of elements of $I$  is called an
{\it{enhanced}}  standard basis of $I$  if the corresponding leading terms generate $\lt{I}.$ Every enhanced standard basis is also a
standard basis, but the converse is not true.
In \cite{singular} an enhanced standard basis of $I$  is simply called a standard basis.
 We have $\lt{I} P =Lt_{\tau}(I^*)$ (see \cite{Ber09} Proposition 1.5.) so that
$$
HF_{R/I}=HF_{P/I^*}=HF_{R/\lt{I}}.
$$
In the Theory of enhanced standard basis a crucial result is the  Grauert's Division theorem, \cite[Theorem 6.4.1]{singular}. It claims the following.
Given a set of  formal power series $f, f_1,\cdots, f_m \in R$
there exist power series $q_1,\dots,q_m,r \in R$ such that $f=\sum_{j=1}^mq_jf_j+r$ and, for all $j=1,\dots,m$,
\begin{enumerate}
\item[(1)]
No monomial of $r$ is divisible by $Lt_{\tauh}(f_j)$,

\item[(2)]
 $Lt_{\tauh}(q_jf_j)\le Lt_{\tauh}(f)$ if $q_j\neq 0.$
\end{enumerate}

\noindent
With the above result we can define  $$NF(f | \{f_1,\dots , f_m\}):=r$$ and obtain in this way a reduced normal form of any power series $f$ with respect to a given finite subset of $R$.
The existence of a reduced normal form is the basis  to obtain, in the formal power series  ring, all  the properties of Gr\"obner basis already proved in the classical case. In particular Buchberger's criterion holds for the power series ring $\series{n}$, see \cite[Theorem 1.7.3]{singular}. A similar approach was  introduced by Mora in 1982 in the localization of $P, $ (see \cite{Mor83a}).

 \smallskip
 We come now to describe the content of the paper.
   The main result  is the description of all the numerical functions which are the Hilbert functions of what we call a quadratic complete intersection of codimension two in $\ser3.$ By this we mean  local rings of type $\ser3/(f,g)$ where $f$ and $g$ are power series of order two which form a regular sequence in $\ser3$ with the property that $g^*\notin (f^*).$

 We first prove in Proposition \ref{base1}  that for the Hilbert function of such local rings with a given multiplicity $e,$  there are only two possibilities:
\begin{enumerate}
\item[(1)]
either is increasing by one up to reach the multiplicity, say
$$\{1,3,4,5,6,7,...,e-1,e,e,e....\},$$

\item[(2)]
or it is increasing by one with a flat in position $n$ which is  unique, by which we mean that for some $n\le e-3$ we have  the sequence
 \begin{center} \begin{tabular}{|c|c|c|c|c|c|c|c|c|c|c|c|c|c|c|}
0 & 1  & 2 & 3& 4& \dots & n-1 & n  &n+1 & n+2 & \dots & e-2 & e-1  & e& \dots  \\ \hline
1 &3 & 4  & 5 & 6 &\dots & n+1 & n+2 & n+2 & n+3 & \dots & e-1  &  e & e& \dots  \\
\end{tabular} \end{center}
\end{enumerate}
\noindent
 It turns out that if the Hilbert function is increasing by one, case (1),  there is no constriction on the multiplicity. Instead, if the Hilbert function has a flat, case (2), the multiplicity $e$ cannot be  too big, namely we must have $e\le 2n.$
  This unexpected result is proved in Theorem \ref{theorem1} which is the main result of this paper.
 Examples \ref{ex} and  \ref{examples2n} show that the above Hilbert functions are realizable.

 We present also  two more results on the Cohen-Macaulayness of the tangent cone of such complete intersections.
 First, in Proposition \ref{increasing1}, we prove that a  quadratic complete intersection of codimension two  in $\ser3$ with  Hilbert function   increasing by one has an associated graded ring which is Cohen-Macaulay.
Finally, as a second application of the methods we used in the proof of the main theorem, we are able to prove in Proposition \ref{xyxz} that for a quadratic complete intersection $A=\ser3/I,$ the tangent cone is Cohen-Macaulay  in the case the vector space $I^*_{\ 2}$  does not contain a square of a linear form.

Section four is devoted to   give a structure theorem, modulo
analytic isomorphisms, of the minimal system of generators of
quadratic complete intersection ideals $I$ of codimension two in $\ser3$, \thmref{conjecture22} and \thmref{conjecture23}.
These   results are obtained by taking into account the two possible Hilbert functions that can  occur
for such an ideal.

In the last section of the paper we give several examples to illustrate our results, as well as possible extensions.

\smallskip
\section{Ideals of  type $(2,b)$}

 From now on we assume that $A=\ser3/I$ where $I$ is a codimension two complete intersection ideal  of $R=\ser3$.
 Given the integers  $b\ge  a \ge  2,$   we say that $A$ is {\it{of type $ (a,b) $}},  or $I$ is of type $ (a,b), $ if $I$ can be generated by a regular sequence $\{f, g\} $ such that \order(f)=a, \order(g)=b  and $g^* \not \in (f^*). $      In the language of \cite[Chapter   III, Section 1]{Hir64a} we  write   $\nu^*(I)=(a,b)$  with the meaning that $I$ is of type $ (a,b).$

 In this paper we will be mainly concerned with local rings of type $(2,2)$; however in this section properties of local rings of type $(2,b)$ will be considered.

  \smallskip
  In the following Proposition we prove that the Hilbert function of a local ring of type $(2,b)$ verifies for every  $n\ge 1$ the inequalities
\begin{equation}
\label{constriction}
0\le HF_A(n+1)-HF_A(n)\le 1.
\end{equation}
The question now is whether every   numerical function $H$ such that $H(0)=1$, $H(1)=3$ and verifies (\ref{constriction})    is the Hilbert function of some local ring of type $(2,b).$ This is not the case because,  for example,  the numerical function $\{ 1,3,4,5,5,6,7,7, ....\} $ verifies (\ref{constriction}) but we will see later  that   it cannot be the Hilbert function of a local ring of type $(2,b).$

\smallskip
A local ring $A$ of type $(2,b)$ is Cohen-Macaulay of embedding dimension three so that  we know that the Hilbert function is not decreasing. We say that $HF_A$ {\it admits a  flat
in position $n$} if $$HF_A(n)=HF_A(n+1) < e.$$

The first basic properties of the Hilbert function of a local ring of type $(2,b)$ are collected in the following proposition which is an easy consequence of the classical Macaulay Theorem.

We recall that  given two positive integer $n$ and $c$,  the $n$-binomial expansion of c is $$c=\binom{c_n}{n}+\binom{c_{n-1}}{n-1}+\cdots \binom{c_j}{j}$$ where $c_n>c_{n-1}>\cdots c_j \ge j\ge 1.$ We let
$$c^{<n>}=\binom{c_n+1}{n+1}+\binom{c_{n-1}+1}{n}+\cdots \binom{c_j+1}{j+1}.$$
The Theorem of Macaulay states that a numerical function $\{h_0,h_1,\cdots,h_i,\cdots,\}$ is the Hilbert function of a standard graded algebra if and only if $h_0=1$ and $h_{i+1}\le h_i^{<i>}$ for every $i\ge 1.$
We remark that if $n+1\le c \le 2n$ then the  $n$-binomial expansion of c is $$c=\binom{n+1}{n}+\binom{n-1}{n-1}+\cdots \binom{2n-c+1}{2n-c+1},$$ so that  $c^{<n>}=c+1.$

Further, if $f_1,\dots,f_r$ are elements of order $d_1,\dots,d_r$ in the regular local ring  $(R,\M)$ and $J$ the ideal they generate,  it is known that $$J^*_n=(J\cap\M^n+\M^{n+1})/\M^{n+1}$$ and $$(f_1^*,\dots,f^*_r)_n=(\sum_{i=1}^r\M^{n-d_i}f_i+\M^{n+1})/\M^{n+1}$$ for every non negative integer $n.$
With this notation we have the following basic lemma.

\smallskip
\begin{lemma}\label{base} Let $I=(f,g)$ be an ideal of $(R,\M)$ with $order(f)=2\le order(g)=b.$ Then
\begin{enumerate}
\item[(i)]  $I^*_{\ j}=(f^*)_j$ for every integer $2\le j<b.$

\item[(ii)] $I^*_{\ b}=(f^*,g^*)_b.$

\item[(iii)] If $g^*\notin (f^*)$ then $I^*_{\ b+1}=(f^*,g^*)_{b+1}.$
\end{enumerate}
\end{lemma}
\begin{proof} Since $j+1\le b$ we have $g\in \M^b\subseteq \M^{j+1}\subseteq \M^j$, hence $$(f,g)\cap\M^j+\M^{j+1}=(g)+(f)\cap\M^j+\M^{j+1}=f\M^{j-2}+\M^{j+1}.$$ The first assertion  follows.
We prove now (ii). We have:$$(f,g)\cap\M^b=(g)+(f)\cap \M^b=(g)+f\M^{b-2}.$$

As for (iii) we need to prove that if $g^*\notin (f^*)$ then $(f,g)\cap\M^{b+1}=f\M^{b-1}+g\M.$ The inclusion $\supseteq$ is clear, so let  $\alpha=cf+dg\in \M^{b+1}.$ If $d\in \M$ then $cf\in \M^{b+1}$ and this implies $c\in \M^{b-1}$ as required. If $d\notin \M$ then $g\in ((f)+\M^{b+1})\cap \M^b=\M^{b+1}+f\M^{b-2}$
which implies  $g^*\in (f^*),$ a contradiction.
\end{proof}

\smallskip
\begin{proposition}
\label{base1}
Let $A=R/I$ be a local ring of
type $(2,b)$ and $I=(f,g)$ with $\order(f)=2$,  $\order(g)=b$ and $g^* \not \in (f^*).$
Then the following properties hold.

\begin{enumerate}

\item[(i)] $HF_A(j)= 2j + 1$  if $j < b.$

\item[(ii)] $HF_A(b)=2b.$

\item[(iii)] $HF_A(j-1)\le HF_A(j)\le HF_A(j-1)+1$ if $j\ge b.$
\item[(iv)] $HF_A $ admits at most $b-1$ flats.
\end{enumerate}
\end{proposition}
\begin{proof} By (i) of the above Lemma we have for every $j<b$ $$HF_A(j)=HF_{P/I^*}(j)=HF_{P/(f^*)}(j)=2j+1.$$

We prove now the second assertion. By (ii) of the above Lemma we have $$ HF_A(b)=HF_{P/I^*}(b)=HF_{P/(f^*,g^*)}(b).$$
Since $g^* \not \in (f^*)$ we get $HF_A(b)=HF_{P/(f^*)}(b)-1=2b+1-1=2b$ as required.

As for (iii) we need only to prove that $HF_A(j)\le HF_A(j-1)+1$ if $j\ge b.$ We have $HF_A(b)=2b,$ $HF_A(b-1)=2b-1,$ hence we can argue by induction on $j.$ Let $j\ge b$ and assuming $HF_A(j)\le HF_A(j-1)+1$ we need to prove that $HF_A(j+1)\le HF_A(j)+1.$

We have $j+1\le HF_A(j)<HF_{P/(f^*)}(j)=2j+1,$ hence, by the remark before the Lemma, we get $$HF_A(j+1)\le HF_A(j)^{<j>}= HF_A(j)+1$$ as wanted.

Finally we prove (iv). We have $HF_A(b)=2b$ and at each step $HF_A$ goes up at most by one. Hence, if there are p flats between b and j, we have $HF_A(j)=2b+j-b-p.$ But $HF_A(j)\ge j+1,$ so that $p\le b-1.$
\end{proof}

\smallskip
 From the above proposition  it follows that the Hilbert function of a local ring of type $(2,b)$ either  is strictly increasing or it  has one or more   flats  (no more than $b-1$); if the first is the case, it has the following shape

\begin{equation} \label{increasing}
 HF_A(j)=
\begin{cases}
          2j+1 & \ \ \text{$j=0, \dots, b-1$}, \\

     j+b & \ \ \text{$b\le j\le e-b,$}\\

	     e & \ \  \text{$ j\ge e-b+1.$}\\
\end{cases}
\end{equation}
where $e$ and $b$ are integers, $b\ge 2$ and $e\ge 2b.$

We show with  the following  example that given a numerical function $H$ as in (\ref{increasing}) we can find a local ring of type $(2,b)$ with multiplicity $e$ whose Hilbert function is $H.$

\begin{example}\label{ex} Let $b\ge 2$ and $e\ge 2b.$ We claim that the above numerical  function is the Hilbert function of the following local ring of type $(2,b)$ and multiplicity $e.$

Let $I=(x^2+y^{e-2b+2},xy^{b-1})$  and  $A=\ser3/I.$ We fix an ordering on the monomials of $P$ with the property that $x>y.$  We let $f:=x^2+y^{e-2b+2},\  g:=xy^{b-1}$ and claim that $\lt{I}=(x^2,xy^{b-1},y^{e-b+1}).$

Since $e\ge 2b$ and $x>y$ it is clear that $\lt{f}=x^2.$ We have $$S(f,g)=y^{b-1}f-xg=y^{b-1}(x^2+y^{e-2b+2})-xxy^{b-1}=y^{e-b+1}.$$  Let $h:=S(f,g)=y^{e-b+1},$  then $$S(f,h)=y^{e-b+1}f-x^2h=y^{e-b+1}( x^2+y^{e-2b+2})-x^2y^{e-b+1}=y^{2e-3b+3}=y^{e-2b+2}h$$ and $S(g,h)=0.$ It follows that
\begin{eqnarray*}
 \NF(S(f,g) \ | \{h\} )  &=&\NF(h \ | \{h\})\;  = \; 0,  \\
  \NF(S(f,h) \ | \{h\} ) &=&\NF(y^{e-2b+2}h \ | \{h\}) \;  = \; 0, \\
  \NF(S(g,h) \ | \{h\} ) &=& \NF(0 \ | \{h\}) \; = \;  0.
\end{eqnarray*}

\noindent
By Buchberger criterion we get that $\lt{I}=(x^2,xy^{b-1},y^{e-b+1})$  as claimed. With the aim of a simple computation we can prove that $K[x,y,z]/(x^2,xy^{b-1},y^{e-b+1})$ has the above Hilbert function; clearly the same is true   for the  local ring $\ser3/(x^2+y^{e-2b+2},xy^{b-1}).$
\end{example}

\smallskip
We end this section by proving that for a local ring of type $(2,b)$ the condition that the Hilbert function is strictly increasing is equivalent to the Cohen-Macaulayness of the tangent cone.
First we need to prove that  the property  of  having  type $(a,b)  $  can be  carried on the quotient  modulo  a suitable superficial  element. We recall that an element  $\ell \in {\mathcal M} $ is superficial for $  {\mathcal M}/I$ if $\ell$ does not belong to any of the associated primes  of $I^*  $ different from  the homogeneous  maximal ideal.
Since the residual field if infinite the existence of superficial elements is guaranteed.
Moreover, it is easy to prove:

\smallskip
\begin{proposition} \label{*}
Let $I$ be an ideal of $R$  of type $(a,b)$ with $2\le a\le b.$
There exists $\ell \in {\mathcal M}\setminus {\mathcal M}^2$  such that
\begin{enumerate}
\item[(i)]
the coset of $\ell$ in $R/I$ is superficial for ${\mathcal M}/I$,

\item[(ii)]
$  \bar I= I+ (\ell)/ (\ell) $ is an ideal of $R/(\ell)$ of type $(a,b).$
\end{enumerate}
\end{proposition}

\begin{proof} It is well known that $\ell $ verifies $(i) $ if   $\ell^*$ does not belong to any of the associated prime ideals  of $I^*$ (different from the homogeneous maximal ideal).  Let   $I=(f,g) $ be with $order (f)=a \le order (g)=b.$ Then it is easy to see that  $\bar I$ satisfies  $(ii) $   provided:

a)  $ \ell^*$ does not divide $f^*$

b)   $ g^* \not \in (f^*, \ell^*)$.

\noindent In fact  $\bar I= (\bar f, \bar g)$  in $R/\ell $ and condition a) assures   $order (\bar f) = a  $ and condition b) gives   $ \bar g ~^* \not \in ( \bar f~ ^*).$   Since $\depth P/(f^*,g^*) \ge 1  $ ($P=\res[x,y,z]$), it is easy to see that for having a) and b)     it is enough to choose $\ell \in {\mathcal M}\setminus {\mathcal M}^2$ such that $\ell^*$ is regular in $P/(f^*, g^*). $
Clearly, if this is the case,   $\ell^* $ does not divide $f^* $ and if  $g^*  \in (f^*, \ell^*), $ then $g^*=\alpha f^* + \beta \ell^* $ with $\alpha, \beta \in P. $ Since  $\ell^*$ is $P/(f^*,g^*)$-regular, then $ \beta \in  (f^*,g^*). $ Hence $g^*= \alpha f^* +\ell^*  (\beta_1 g^* + \beta_2 \ell^*), $ so $g^*(1-\ell\beta_1) \in (f^*), $ a contradiction because $ g^* \not \in (f^*). $
   Since  the residue field is infinite, an element $\ell \in {\mathcal M}\setminus {\mathcal M}^2$ verifying  the conditions of the proposition  can be selected   by avoiding the  associated prime ideals to $I^*$ and to $(f^*, g^*). $
\end{proof}

\smallskip
 It is well known  that if the associated graded ring $gr_{\m}(A) $ is Cohen-Macaulay,  then
  the Hilbert function  of $A$  is strictly increasing. However the converse is in general very rare. In the following result we will show a special case where this implication holds true.

\smallskip
 \begin{proposition}
\label{increasing1}
Let $A=R/I$ be a local ring of
type $(2,b).$ Then $gr_{\m}(A) $ is Cohen-Macaulay if and only if $HF_A$ is strictly increasing.
\end{proposition}
\begin{proof}
 Let   $I=(f,g)$ with $\order(f)=2, $ $\order(g)=b $ and $g^* \not \in (f^*).$ If the associated graded ring is Cohen-Macaulay,  then  its Hilbert function is strictly increasing and thus the Hilbert function of $A$ is strictly increasing as well.
Assume   that $HF_A$ is strictly increasing.
From Proposition  \ref{base1},    a simple computation gives
$$
\Delta HF_A(n):=HF_A(n+1)-HF_A(n)=
\begin{cases}
     1 & \ \ \text{$n=0$}, \\
     2 & \ \ \text{$n=1, \dots, b-1$}, \\
     1 & \ \ \text{$n=b,\dots, r-1$},\\
     0 & \ \ \text{$n \ge r$}\\
\end{cases}
$$
with $r=e-b+1$.

From \propref{*}  there exists a  superficial element $x\in A$ such that
$$
HF_{A/xA}(n)=
\begin{cases}
     1 & \ \ \text{$n=0$}, \\
     2 & \ \ \text{$n=1, \dots, b-1$}, \\
     1 & \ \ \text{$n=b$}.\\
\end{cases}
$$
From Macaulay's characterization of Hilbert functions and the fact that $e(A/xA)=e(A), $
we get $\Delta HF_A = HF_{A/xA}$.
 Hence
$gr_{\m}(A)$ is Cohen-Macaulay, \cite{Sin74}.
\end{proof}

\smallskip
Notice that the above proposition cannot be extended to local rings of type $(a,b) $ with $a>2,$  as  the following  example shows.
Consider the  local ring $A=R/I $ where $ I=(x^4, x^2y+z^4) \subseteq R=\ser3;$   $A$ is a  one-dimensional Gorenstein local ring and  $$HF_A=\{1,3,6,9,11,13,14,15,16,16,\dots ,\} $$ is  strictly increasing. Now it is clear that $x^4,x^2y\in I^*$  and since $x^2(x^2y+z^4)-yx^4\in I,$ also $x^2z^4\in I^*.$ This implies that $x^3z^3 (x,y,z)\subseteq I^*;$ since $x^3z^3\notin I^*$ $gr_{\m}(A) $ is not Cohen-Macaulay.

 A natural and general problem  would be to characterize the Hilbert functions of all the ideals  $I$ of type $(2,b) . $ If  the Hilbert function has one or more  flat, the behavior is difficult to control. However if we    denote by $p$ the number of  flats, by Proposition \ref{base1} we know that   $p \le b-1. $ With the aid of huge computations made with CoCoa, we ask the following question.

\smallskip
\begin{question}

Let $A=R/I$ be a local ring of
type $(2,b)$
 with $b \ge 2 $ and multiplicity $e. $  Let $ n := \min\{ j: HF_{R/I} (j) =HF_{R/I}(j+1)< e \}  $ and let $p$ be the number of flats. Then
$$ e \le (p+1) n\ ( \le bn).$$
\end{question}
The main result of the paper answer   the question   in the case $a=b=2.$

\smallskip
\section{The main result}

In this section we present a complete characterization of the numerical functions which are the Hilbert functions   of  local rings  of   type $(2,2). $   In particular we  prove that certain monomial  ideals cannot be the initial ideals of a complete intersection, a relevant  task  even in the graded setting (see for example \cite{CD05}).

By the definition we gave in the above section, a local ring  $A=\ser3/I$ of type $(2,2)$ is of dimension one, $HF_A(1)=3$ and  $HF_A(2)=4.$  In particular   $I$ can be generated by a regular sequence, say $I=(f,g),$  where $f$ and $g$ are power series of order two such that $f^*$ and $g^*$ are linearly independent in the vector space $K[x,y,z]_2.$ We recall that by Proposition \ref{base1}  we have $I^*_{\ 2}=<f^*,g^*>$ and  $I^*_{\ 3}=(f^*,g^*)_3=<f^*x,f^*y,f^*z,g^*x,g^*y,g^*z>.$

Since $HF_A(2)=4$ we know that $4\le HF_A(3)\le HF_A(2)^{<2>}=5.$ If  $HF_A(3)=4,$ then $6=dim(I^*_{\ 3})=\dim_\res <f^*x,f^*y,f^*z,g^*x,g^*y,g^*z>.$ This easily implies that $f^*$ and $g^*$ form a regular sequence in $K[x,y,z].$ As a consequence  $I^*=(f^*,g^*)$ and the Hilbert function of $A$  is $\{1,3,4,4,4,....\}$ which is as in (\ref{increasing}) with $b=2,$ $e=4.$

We want to study the remaining case, when $HF_A(3)=5.$ We first remark that in this case $f^*$ and $g^*$ share a common factor, say $L,$ which must be a linear form because $f^*$ and $g^*$ are linearly independent. Hence we can write $$f^*=LM,\ \ \ \ g^*=LN$$ where $L,M,N$ are linear forms in $K[x,y,z]$
such that $M$ and $N$ are linearly independent. In particular $I^*_{\ 2}=<LM,LN>.$

We have two possibilities, either $L,M,N$ are linearly independent or are linearly dependent.
We remark that this property depends on the ideal $I$ and not on the generators of $I$.
Namely, if we say that $I^*_{\ 2}$ is {\it{square free}}  with the meaning that  it does not contain a square of a linear form, we can prove the following easy result:

\smallskip
 \begin{lemma} With the above assumption, the vectors $L,M,N$ are linearly independent if and only if   $I^*_{\ 2}$ is square-free.
 \end{lemma}
	  \begin{proof} Let us first assume   that $L,M,N$ are linearly dependent. Since $M,N$ are linearly independent we have $L=\alpha M+ \beta N$ so that $L^2=\alpha LM+ \beta LN\in  <I^*_{\ 2}>.$ Hence $I^*_{\ 2}$ is not square-free.

 We prove now that if  $L,M,N$ are linearly independent then $I^*_{\ 2}$ is  square-free. Let $P$ be a linear form such that $P^2 \in I^*_{\ 2}=<LM,LN>$; then $P\in (L)$ so that $P=\lambda L.$ We have $\lambda^2 L^2=\alpha LM+ \beta LN$ hence $\lambda^2 L=\alpha M+ \beta N;$ since $L,M,N$ are linearly independent this implies $\lambda =0$ and finally $P=0.$
  \end{proof}

\smallskip
For completeness, we need now to recall the notion of k-algebra isomorphism. Given a set of minimal generators $\underline{y}=\{y_1,y_2,...,y_n\}$ of the maximal ideal of $R=\series{n}$, we let $\phi_{\underline{y}}$ be the automorphism of $R$ which is the result    of  substituting $y_i$ for $x_i$  in  a power series $f(x_1,x_2,...,x_n)\in R.$
 Given two ideals $I$ and $J$ in $R$ it is well known that there exist a $K$-algebra isomorphism $\alpha : R/I\to R/J$ if and only if for some generators $y_1,y_2,...,y_n$ of the maximal ideal  of $R$, we have $I=\phi_{\underline{y}}(J).$

  \vskip 2mm We start now by   deforming, up to isomorphism, the generators $f$ and $g$ of the given ideal $I.$

\smallskip
\begin{lemma}
\label{initial}
Let $A=R/I$ be a local ring of  type $(2,2)$ such that $HF_A(3)=5.$
\vskip 3mm
\noindent
$(i)$
If $I^*_{\ 2}$ is not square-free we may assume, up to isomorphism, that $I=(f,g)$ with $f^*=x^2$ and $g^*=x y$.

\vskip 3mm
\noindent
$(ii)$
 If $I^*_{\ 2}$ is  square-free we may assume, up to isomorphism, that $I=(f,g)$ with $f^*=xy$ and $g^*=x z$,
\end{lemma}
\begin{proof}  Let us first assume that $I^*_2$ is not square-free; then $f^*=LM,$ $g^*=LN$ with
$L,M,N$  linearly dependent; since $M$ and $N$ are linearly independent, we must have  $L=\lambda M+\rho N$ for suitable $\lambda$ and $\rho $ in $K$ with $(\lambda, \rho )\neq (0,0).$ By symmetry we may assume $\lambda \neq 0.$  Then it is easy to see that $L$ and $N$ are linearly independent so that we can consider an automorphism $\phi$   sending $x \to L, y\to N.$  We have  $f=LM+a$ and $g=LN+b$ for suitable $a,b \in  \mathcal M^3,$ and further $$L^2=\lambda LM+\rho LN=\lambda f+\rho g-\lambda a-\rho b.$$
We get    $$I=(f,g)=(\lambda f,g)=(L^2-\rho g+\lambda a+\rho b,g)=(L^2+\lambda a+\rho b,g)=$$
$$=(L^2+\lambda a+\rho b,LN+b)
=\phi((x^2+\phi^{-1}(\rho b+\lambda a),xy+\phi^{-1}(b)).$$   The conclusion follows.

Now we assume  that $I^*_{\ 2}$ is  square-free. Then $f^*=LM$ and $g^*=LN$
  where $L, M, N$ are linear forms in  $K[x,y,z]$ which are linearly independent. As before we have    $f=LM+a$ and $g=LN+b$ for suitable $a,b \in  \mathcal M^3,$

 Let us  consider the automorphism $\phi$ sending $x \to L, y \to M, z \to N.$ We have $$I=(f,g)=(LM+a,LN+b)=\phi((xy+\phi^{-1}(a),xz+\phi^{-1}(b)))$$
 and the conclusion follows. \end{proof}

\smallskip
Using Grauert division theorem, we can prove a first useful preparation result in the case $x^2= \lt f$ and $xy=  \lt g.$

\smallskip
 \begin{lemma}
 \label{Grauert}
 Let $A=R/I$ be a local ring of
type $(2,2)$ such that  $I=(f,g),$ $ \lt f=x^2,$  $  \lt g=xy.$  Then we can write
$$
I=(x^2 + a x z^p + F(y,z), xy + b x z^q+ G(y,z))
$$
where $p, q \ge 1$, $a =0$ or $a \in  \mathbb U(\serz),  $   $b =0$ or $b \in  \mathbb U(\serz),  $
$F, G\in \seryz_{\ge 2}.$
\end{lemma}
\begin{proof}
By the assumption  we have   $ f=x^2+ F$ with $ \lt{F}<_{\tauh} x^2$ and $ g=xy +G $ with $ \lt{G}<_{\tauh} xy. $
Applying Grauert's division theorem to the power series $F, f, g$ we get
$$
F= \alpha f + \beta g +r
$$
where $\alpha, \beta, r\in R$, no monomial of $\supp(r)$ is divisible by $x^2$ or $xy$, and
$$
\lt{\alpha f},  \lt{\beta g} \le_{\tauh} \lt{F} <_{\tauh} \lt{f}=x^2.
$$
We can write $\alpha=\sum_{i\ge 0} \alpha_i,$ where, for every $i,$ $\alpha_i $ is a degree $i$ form in $\ser3.$ It is clear that the initial form of $\alpha f=\alpha (x^2+F)$ is $\alpha_0x^2+\alpha_0F_2$ so that
 $\alpha_0= 0,$ otherwise  $
\lt{\alpha f}   =x^2.
$ In particular $1-\alpha$ is a unit. Since

$$(1-\alpha)f =f-\alpha f=x^2+F-\alpha f=  x^2+ r+ \beta g$$
 we get $$
I=(f,g)=(x^2 +r ,g).
$$

We apply now Grauert's Division Theorem to the power series
$G, x^2 +r, f$ where $G=g-xy$ and $\lt{G}<_{\tauh} xy.$ We get
$$g-xy=G=t(x^2+r)+sg+r'$$ where no monomial of $\supp(r')$ is divisible by $\lt{x^2+r}= \lt{x^2+F-\alpha f -\beta g}=x^2$ or by $\lt{g}=xy.$

Since $g=xy+t(x^2+r)+sg+r',$ we get $$g(1-s)=t(x^2+r)+r'+xy$$ and we claim that $1-s$ is a unit. Namely,    $\lt{sg}\le \lt{G}<xy$ and, as before, $$sg=s(xy+G)=s_0(xy+G)+s_1(xy+G)+....$$ This implies  $s_0= 0,$ otherwise  $
\lt{sg}   =xy.$ This proves the claim.

 Now we have $$I=(x^2+r,g)=(x^2+r,(1-s)g)=(x^2+r, t(x^2+r)+r'+xy)=(x^2+r,xy+r'),$$ where  no monomial of $\supp(r)$ and $\supp(r')$ is divisible by $x^2$ or $xy$.

It is easy to see that this implies
\begin{eqnarray*}
  r  &=&  a x z^p + F(y,z)\\
  r' &=&  b x z^q +G(y,z)
\end{eqnarray*}
with $p,q\ge 1$, $F, G\in \seryz_{\ge 2}$,
$a=0$ or $a \in \mathbb U(\serz)$, and
$b=0$ or $b \in \mathbb U(\serz)$.
 \end{proof}

\smallskip
We can prove now the main preparation result.

\smallskip
\begin{theorem}
\label{preslemma}
Let $A=R/I$ be a local ring of
type $(2,2)$ such that $HF_A(3)=5.$
\vskip 2mm \noindent
a) If $I^*_2$ is not square-free then, up to isomorphism, we can write
$$
I=(x^2 + a x z^p + F(y,z), xy + G(y,z))
$$
where $p\ge 2$, $a\in\{0,1\}$,   and $F, G\in \seryz_{\ge 3}.$
\vskip 2mm \noindent
b) If $I^*_2$ is  square-free then, up to isomorphism, we can write
$$I=(x^2 +   x z + F(y,z), x y +  d y z+ \alpha y^r + \beta z^s)$$
where
$F\in \seryz_{\ge 3}, $
$d\in \seryz,$  $d(0,0)=1$,
 $r, s\ge 3$,
$\alpha=0$ or $\alpha \in \mathbb U(\sery)$,
$\beta=0$ or $\beta \in \mathbb U(\serz)$.

\end{theorem}
\begin{proof} By Lemma \ref{initial}, up to isomorphism,  we can find generators $f$ and $g$ of $I$  such that  either $f^*=x^2$ and $g^*=xy$ or
$f^*=xy$ and $g^*=xz.$

Let us first assume that
  $f^*=x^2$ and $g^*=xy$; then $$\lt{f}=Lt_{\tau}(f^*)=Lt_{\tau}(x^2)=x^2$$
 $$\lt{g}=Lt_{\tau}(g^*)=Lt_{\tau}(xy)=xy,$$ so that, as remarked at the end of the proof of  Lemma \ref{Grauert}, we have $I=(x^2+r,xy+r')$ where  no monomial of $\supp(r)$ and $\supp(r')$ is divisible by $x^2$ or $xy.$

 Since $f^*=x^2$ and $g^*=xy, $ we also have $f=x^2+h,$ $g=xy+s$ where $\order(h), \order(s)\ge 3.$
 This implies  $I=(x^2+r,xy+r')=(x^2+h,xy+s)$ and $I^*_2=<x^2,xy>$, the vector space spanned by $x^2$ and $xy.$. Since  the degree 2 component of $r$  is a linear combination of the monomials $xz,y^2,yz,z^2,$ it must be zero, otherwise the leading form of $x^2+r$ cannot be in $I^*_2=<x^2,xy>.$ This proves that  the order of $r$ is at least 3. Exactly in the same way we can prove that this holds  true also for   $r'$.

 It is easy to see that this implies $$r=axz^p+D(y,z),\ \ \ \ \ r'=bxz^q+E(y,z)$$ where
 $p, q \ge 2$, $a =0$ or $a \in  \mathbb U(\serz),$   $b =0$ or $b \in  \mathbb U(\serz),$ and
$D, E\in \seryz_{\ge 3}.$

Now let $\phi$ be  the  automorphism of $\ser3$  defined by
$$x\rightarrow x,\ \ y\rightarrow y -  b  z^q, \ \ z\rightarrow z$$ and let $S:=\phi(D)$ and $T:=\phi(E).$ Then $S, T\in \seryz_{\ge 3}$ and we have  $$\phi(f)=\phi(x^2+r)=\phi(x^2+axz^p+D(y,z))=x^2+axz^p+S(y,z))$$ and $$\phi(g)=\phi(xy+r')=\phi(xy+bxz^q+E(y,z))=x(y-bz^q)+bxz^q+\phi(E(y,z))=xy+T(y,z)).$$

 This implies that, up to isomorphism, we may assume
$$I=(x^2+axz^p+S(y,z),xy+T(y,z))$$ with $p \ge 2$, $a =0$ or $a \in  \mathbb U(\serz),$   $b =0$ or $b \in  \mathbb U(\serz),$ and
$S, T\in \seryz_{\ge 3}.$

Now if $a=0$ we are done, otherwise let $a\neq 0.$ Since the ground field  $\res$ is algebraically closed and $a$ is invertible in $K[[z]]$, a straightforward application of Hensel Lemma  enables us to find an element $c \in R$
 such that $c^p=a.$

 Let us consider the automorphism $\phi:\ser3\to \ser3$  defined by
$$x\rightarrow x,\ \ y\rightarrow y, \ \ z\rightarrow cz.$$ If $F$ and $G$ are  power series in $K[[z]]$ such $\phi(F)=S$ and $\phi(G)=T,$ then
$$\phi(x^2+xz^p+F)=x^2+xc^pz^p+S=x^2+axz^p+S,\  \  \ \ \ \ \ \phi(xy+G)=xy+T.$$
The conclusion easily follows.

We need now to  consider the other case when $f^*=xy,$ $g^*=xz.$ As before we choose  a  monomial order $\tau $ such that  $x >_{\tau}z$ and let $\phi$ be  the  automorphism of $\ser3$  defined by
$$x\rightarrow x+z,\ \ y\rightarrow x, \ \ z\rightarrow y.$$ We have $f=xy+d,$ $g=xz+e$ where $d$ and $e$ have order at least 3. Hence $$\phi(f)=(x+z)x+\phi(d)=x^2+xz+h,\ \ \ \ \phi(g)=(x+z)y+\phi(e)=xy+yz+s$$
where $h:=\phi(d)$ and $s:=\phi(e)$ have order $\ge 3.$ Thus, up to isomorphism, we may assume that $I$ is generated by the power series $
x^2+xz+h$ and $xy+yz+s;$ this implies that $ I^*_2=<x^2+xz,xy+yz>.$

Since $x^2>_{\tau}xz$ and $xy>_{\tau}yz,$ we get
$$\lt{x^2+xz+h}=Lt_{\tau}((x^2+xz+h)^*)=Lt_{\tau}(x^2+xz)=x^2$$  $$\lt{xy+yz+s}=Lt_{\tau}((xy+yz+s)^*)=Lt_{\tau}(xy+yz)=xy$$ and we may use  Lemma \ref{Grauert} to get $$I=(x^2+axz^p+S(y,z),xy+bxz^q+M(y,z))$$ where   $p, q \ge 1$, $a =0$ or $a \in  \mathbb U(\serz),  $   $b =0$ or $b \in  \mathbb U(\serz),  $
$S, M\in \seryz_{\ge 2}.$

 Now let $\alpha:=x^2+axz^p+S(y,z);$ if $p\ge 2$ then
 $\alpha^*=x^2+S(y,z)_2$ is an element of the vector space $ I^*_2=<x^2+xz,xy+yz>,$ a contradiction.  Hence $p=1$ and thus we get $\alpha^*=x^2+a_0xz+S(y,z)_2\in <x^2+xz,xy+yz>.$  This clearly implies $a_0=1$ and $S(y,z)_2=0$, so that the order of $S(y,z)$ is at least 3.

 Now let $\beta:=xy+bxz^q+M(y,z);$ if  $b\neq 0$ and $q=1$ then $b_0\neq 0$ and we have $$\beta^*=xy+b_0xz+M(y,z)_2\in I^*_2=<x^2+xz,xy+yz>,$$ a contradiction. Hence it must be either $b=0$ or $q\ge 2;$ in both cases we have $$\beta^*=xy+M(y,z)_2\in <x^2+xz,xy+yz>$$ which implies $M(y,z)=yz+H(y,z)$ where $H(y,z)$ is a power series in $K[[y,z]]$ with  order at least 3.

 At this point we have $I=(x^2+axz+S(y,z),xy+bxz^q+yz+H(y,z))$ with $a_0=1,$ S and H $\in K[[y,z]]_{\ge 3}$  and either  $b=0$ or $q\ge 2$ .

 Let us consider the  automorphism $\phi$ given by $$x \to x, \ \ y\to y-bz^q,  \ \ z\to z.$$ We get $$
\phi(x^2+axz+S(y,z))=x^2+axz+B(y,z),$$ and
$$\phi(xy+bxz^q+yz+H(y,z))=x(y-bz^q)+bxz^q+(y-bz^q)z+\phi(H)=xy+yz+L(y,z)$$ where $B(y,z):=\phi(S)$ and $L(y,z):=-bz^{q+1}+\phi(H)\in K[[y,z]]_{\ge 3}.$

Hence, up to isomorphism, we may assume $I=(x^2+axz+B(y,z),xy+yz+L(y,z))$ with $a_0=1$ and  $B,L\in K[[y,z]]_{\ge 3}.$ Now it is clear that since $L(y,z)$ has order at least 3, we can write
 $L(y,z)= c  yz +   \alpha y^r + \beta z^s $ with  $c \in K[y,z]_{\ge 1}, $  $r, s\ge 3$,
$\alpha=0$ or $\alpha \in \mathbb U(\sery)$, and
$\beta=0$ or $\beta \in \mathbb U(\serz)$.   Hence we get $I=(x^2+axz+B(y,z), xy+ yz+ c  yz +   \alpha y^r + \beta z^s).$  We let $d:=1+c$ so that $d\in K[[y,z]],$ $d(0,0)=1+c(0,0)=1$ and $$I=(x^2+axz+B(y,z), xy+  dyz +   \alpha y^r + \beta z^s).$$ Finally let us consider the automorphism $\phi$ given by $$x \to x, \ \ y\to y,  \ \ z\to az.$$ Let $F(y,z):=\phi^{-1}(B(y,z))$ and $$f:=x^2+xz+F(y,z)\ \ \ \  g:=xy+\phi^{-1}(d/a)yz+\alpha y^r+\phi^{-1}(\beta/a^s)z^s.$$ Then we get $$\phi(f)=x^2+axz+B(y,z)$$ $$ \phi(g)=xy+(d/a)yaz+\alpha y^r+(\beta/a^s)(a^sz^s)=xy+  dyz +   \alpha y^r + \beta z^s.$$ We remark  that the constant term of the power series $d/a$ is 1 and the power series $\beta/a^s$ is invertible if not zero.  Hence the same holds for $\phi^{-1}(d/a)$ and $\phi^{-1}(\beta/a^s).$ The conclusion follows.
\end{proof}

\smallskip
We recall  that in this section we are assuming that $A=\ser3/I$ is a local ring of type (2,2) such that $HF_A(1)=3,$ $HF_A(2)=4$ and $HF_A(3)=5.$
This implies that if we let $n$ to be  {\bf{the least integer such that}} $HF_A(n)=HF_A(n+1),$ then $n\ge 3.$
Also it is easy to see that $n\le r,$ the reduction number of $A$. The integer $n$ plays a relevant work in the paper.
With the aid of this integer  $n$  and as a consequence of Proposition \ref{base1},   the Hilbert function of a local ring $A$ of type (2,2) and multiplicity $e$ has the following shape:
 \begin{equation} \label{flat} HF_A(t)=
\begin{cases}
     1 & \ \ \text{$t=0$}, \\
     t+2 & \ \ \text{$t=1,\cdots, n$},\\
     t+1 & \ \ \text{$t=n+1,\cdots, e-1$},\\
     e & \ \ \text{$t\ge e$}.\\
\end{cases}
\end{equation}
for some integer $n\le e-2.$
We say that $HF_A$ has a {\bf flat} in position $n$.

It is clear that we have two possibilities, either $e=n+2$ or $e\ge n+3.$ In the first case
   the Hilbert function is  increasing by one up to reach the multiplicity, while in the second case   the Hilbert function has a flat in position $n$ and is increasing by one in all the other positions, before reaching the multiplicity.

We are ready to prove  the main result of this paper. It says that, quite unexpectedly, if the Hilbert function of a local ring of type $(2,2)$ has a flat in position $n,$ then the multiplicity cannot be too big, namely it cannot overcome  $2n.$

First we need this easy Lemma.

\smallskip
\begin{lemma}
\label{preparation}
Let $J\subset P=k[x,y,z] $ be a monomial ideal such that
$x^2, xy\in J.$
If  for some $n\ge 2$  we have $\ HF_{P/J}(n+1)=n+2\  $ and   $ HF_{P/J}(n~+~2)=~n+3,$
 then $x z^n$ is the unique monomial of degree $n+1$ which is in $J$ and not in $(x^2,xy).$

  If we have also $\ HF_{P/J}(n)=n+2, $ then  $J_d=(x^2 , xy)_d$ for all $2\le d\le n$.
 \end{lemma}
\begin{proof}
 Since  $HF_{P/J} (n+1)=n+2< HF_{P/(x^2,xy)}(n+1) =n+3$  there is a monomial $m$ of degree n+1 which is in $J$ and not in $(x^2,xy).$ If $m\neq xz^n$ it should be $m=y^{n+1-j}z^j$ for some $j=0,...,n+1.$ But then the monomials of the vector space $(x^2,xy)_{n+2}$ and $my,mz$ would be linearly independent.
 This implies that $$n+3=HF_{P/J}(n+2)\le HF_{P/(x^2,xy,my,mz)}(n+2)=HF_{P/(x^2,xy)}(n+2)-2=n+2,$$ a contradiction. Hence $m=xz^n.$

 Let us assume that also $\ HF_{P/J}(n)=n+2;$ if for some $t\le n-1$ we have $HF_{P/J}(t)\le t+1,$ then $$HF_{P/J}(t+1)\le HF_{P/J}(t)^{<t>}\le (t+1)^{<t>}=t+2$$ and going on in this way we would have $HF_{P/J}(n)\le n+1,$ a contradiction. It follows that for all $2\le d\le n$ we have $HF_{P/J}(d)=HF_{P/(x^2,xy)}(d)$ and, as a consequence,
 $J_d=(x^2 , xy)_d$ for the same $d$'s.
\end{proof}

\smallskip
\begin{theorem}
\label{theorem1}

Let $A$ be a local ring of type (2,2) and  multiplicity $e$.
  If the Hilbert function of $A$ has a flat in position $n,$  then $e\le 2n.$
\end{theorem}
\begin{proof}   As usual, we consider a monomial ordering $\tau$ on the terms of $K[x,y,z] $ such that $ x>_{\tau}z.$ In order to cover both case a) and b) in   Theorem \ref{preslemma},  we may assume   $I=(f,g)$  where
$$
f:=x^2 + a x z^p + F(y,z) \ \ \ \  g:=xy +G(y,z)
$$ are power series such that $p\ge 1$ and $a=0$ or $a\in \mathbb U(\serz).$

Since it is clear that  $(x^2,xy)\nsubseteq \lt{ I}$, the elements $f$ and $g$ are not a standard basis for $I;$
thus, by Buchberger's criterion,   we should have
$$
h:=\NF(S(f,g),\{f,g\})\neq 0.
$$
It is clear that $h\in I$ and if we let   $m:=\lt{h}=\rm{Lt}_{\tau}(h^*)$, then $m\in \lt{I}$ and by 1.6.4 in \cite{Sin74}  $m    \notin (x^2,xy).$ We claim that $m$ is a monomial of degree $n+1.$

 Namely, by  the second statement of \lemref{preparation} applied to the monomial ideal $\lt{I},$ it is clear that $m$ has degree at least $n+1.$ Let us assume that $m$ has degree $\ge n+2,$ so that $\order(h)\ge n+2.$ Since for every  $s$ and $G$  one can easily prove  that $\order(s)\le \order(\NF( s\  |  G)),$ we get  $$\order(\NF(S(f,h) | \{f,g,h\}))\ge \order(S(f,h))\ge max\{\order(f),\order(h)\}\ge n+2.$$ In the same way we  can also prove that $\order(\NF(S(g,h) | \{f,g,h\}))\ge  n+2.$ Now recall that, accordingly to the  Buchberger algorithm, in order to determine  a standard basis of  $\lt{I}$, one has to compute  $\NF(S(f,h) | \{f,g,h\}),$   $\NF(S(g,h) | \{f,g,h\}),$  to add those of them which are not zero to the list and go on  in this way  up to the end. At each step of this procedure  the order of the elements can only increase; hence   if $m$ has degree $\ge n+2$ then $\NF(S(f,h) | \{f,g,h\})$ and $\NF(S(g,h) | \{f,g,h\})$ have degree at least $n+2$ and we cannot obtain, as \lemref{preparation} requires, the monomial $xz^n$ which has degree $n+1.$ This proves the claim.
 By \lemref{preparation} the claim implies that $m=\lt{h}={\rm{Lt}}_{\tau}(h^*)=xz^n.$

We want now to compute  $\NF(S(f,g)\ | \ \{f,g\})$. First it is clear that we can write $G(y,z)=yH(y,z)+\alpha z^c$  with $\alpha =0$ or invertible in $K[[z]]$ and  $c\ge 0.$ Hence $I=(f,g)$ where $$f=x^2+axz^p+F(y,z),\ \ \ g=xy+y H(y,z)+\alpha z^c$$  with $c\ge 0,$ $p\ge 1$ and $a$  and $\alpha$ either zero or invertible in $K[[z]].$ We have
$$
S(f,g)= y f - x g= a x y z^p + y F(y,z)- xy H(y,z)-\alpha xz^c=g(az^p-H(y,z))-\alpha xz^c+M(y,z)$$
where $M(y,z)=yF(y,z)-(yH(y,z)+\alpha z^c)(az^p-H(y,z))\in K[[y,z]].$

\vskip 1mm We claim that $\NF(S(f,g)\ | \ \{f,g\})=M(y,z)-\alpha xz^c.$ Namely we have $$S(f,g)=0 \cdot f+(az^p-H(y,z))g+M(y,z)-\alpha xz^c$$  and we need to prove:

\vskip 1mm a)  no monomial in the support of $M(y,z)-\alpha xz^c$ is divisible by $x^2$ or $xy$

\vskip 1mm b) $\lt{g(az^p-H(y,z))}\le \lt{S(f,g)}.$

\vskip 2 mm Now a) is true because $\alpha$ is zero or invertible in $K[[z]].$ As for b) it is clear that we have  $\lt{g(az^p-H(y,z))}=xy\cdot \lt{az^p-H(y,z)}.$ This monomial is not in the support of $M(y,z)-\alpha xz^c,$ hence it is in the support of
$S(f,g).$  This implies b) and the claim
\begin{equation}\label{NFh}
h=\NF(S(f,g)\ | \ \{f,g\})= M(y,z) - \alpha x z^c
\end{equation} is proved.

\noindent
Since $\lt{h}= x z ^n$ it follows $\alpha \in \mathbb U(\serz)$, $c=n$, and $\order(M)\ge n+1$.
In particular we deduce
\begin{equation} \label{g}
g = xy + \alpha z^n + y H(y,z).
\end{equation}

Let  $J:=I+(y) = (x^2 + a xz^p +F(0, z), z^n, y);  $ it is clear  that $\lt J  \supseteq (x^2, z^n, y). $ Since $R/J$ is  Artinian, $\overline{y}$ is a parameter in $A= R/I;$ hence  $$ e= e(R/I) \le length(R/J) = length (R/ \lt J \le length (R/(x^2, z^n, y))= 2n.$$ The conclusion follows.
\end{proof}

\smallskip
In example \ref{ex}, we have seen that however we fix an integer $e \ge 4, $ there is a local ring of type $(2,2) $ with multiplicity $e$ and strictly increasing Hilbert  function.   For each pair of integers $(n,e) $  such that $n\ge 3$ and $n+3 \le e\le 2n,$ we exhibit  now  local rings of type $ (2,2)$ and multiplicity $e$ whose Hilbert function has  a flat in position $n.$

\smallskip
\begin{example}
\label{examples2n}
Given the  integers $n$ and $e$ such that $n\ge 3$,  $n+3 \le e \le 2n,$ the ideal $$I=(x^2-y^{e-2}, xy-z^n)$$
is a complete intersection ideal of $R=\ser3$ of type $(2,2) $ with multiplicity $e,$  whose Hilbert function has a flat in position $n.$
\end{example}
\begin{proof}
Let us consider  a monomial ordering $\tau$ such that $x>y>z; $   we are going  to prove that $$\{f=x^2-y^{e-2},  \ \   g=xy-z^n, \ \  h=-y^{e-1}+ x z^n, \ \   k=y^e-z^{2n}\}$$ is a standard basis for $I$. Namely, if this is the case, we get
$\lt I=(x^2, xy, xz^n, y^e)$ and from this an easy computation shows that  the local ring $K[x,y,z]]/(x^2-y^{e-2}, xy-z^n)$has multiplicity $e$ and Hilbert function with a flat in position $n.$

We have: $$S(f,g)=yf-xg=y(x^2-y^{e-2})-x(xy-z^n)=xz^n-y^{e-1}$$ and since $e\ge n+3$ implies $e-1\ge n+2>n+1$, we get  $\lt{S(f,g)}=xz^n.$

We let $$h:=S(f,g)=xz^n-y^{e-1}.$$

Now $$S(f,h)=z^nf-xh=z^n(x^2-y^{e-2})-x(xz^n-y^{e-1})=xy^{e-1}-z^ny^{e-2}=y^{e-2}g$$ so that $\lt{S(f,h)}=y^{e-2}\lt{g}=xy^{e-1}.$

Further $$S(g,h)=z^ng-yh=z^n(xy-z^n)-y(xz^n-y^{e-1})=y^e-z^{2n}$$ and since $e\le 2n$ and $y>z,$  we have $\lt{S(g,h)}=y^e.$

We let $$k:=S(g,h)=y^e-z^{2n}$$ with $\lt{S(g,h)}=\lt{k}=y^e.$

Now  $$S(f,k)=y^ef-x^2k=y^e(x^2-y^{e-2})-x^2(y^2-z^{2n})=x^2z^{2n}-y^{2e-2}=z^{2n}f-y^{e-2}k$$ and since
$2e-2\ge 2(n+3)-2=2n+4>2n+2,$  we have $\lt{S(f,k)}=x^2z^{2n}.$

Also $$S(g,k)=y^{e-1}g-xk=y^{e-1}(xy-z^n)-x(y^2-z^{2n})=xz^{2n}-y^{e-1}z^n=z^nh$$ so that $\lt{S(g,k)}=z^n\lt{h}=xz^{2n}.$

Finally $$S(h,k)=y^2h-xz^nk=y^e(xz^n-y^{e-1})-xz^n(y^e-z^{2n})=xz^{3n}-y^{2e-1}=z^{2n}h-y^{e-1}k.$$ Here we can only remark that $\lt{S(h,k)}=\max\{xz^{3n},y^{2e-1}\}.$

From these computations we get
$$\NF(S(f,g)\ | \ \{h\})=\NF(h\ | \ \{h\})=0$$
$$\NF(S(f,h)\ | \ \{g\})=\NF(y^{e-2}g\ | \ \{g\})=0$$
$$\NF(S(g,h)\ | \ \{k\})=\NF(k\ | \ \{k\})=0$$
$$\NF(S(f,k)\ | \ \{f,k\})=\NF(z^{2n}f-y^{e-2}k\ | \ \{f,k\})=0$$
because $\lt{z^{2n}f-y^{e-2}k}=x^2z^{2n}\ge \lt{z^{2n}f}=x^2z^{2n},   \lt{y^{e-2}k}=y^{2e-2}.$
$$\NF(S(g,k)\ | \ \{h\})=\NF(z^nh\ | \ \{h\})=0$$
$$\NF(S(h,k)\ | \ \{h,k\})=\NF(z^{2n}h-y^{e-1}k\ | \ \{h,k\})=0$$ because
$\lt{z^{2n}h-y^{e-1}k}=\max\{xz^{3n},y^{2e-1}\}\ge \lt{z^{2n}h}=xz^{3n},   \lt{y^{e-1}k}=y^{2e-1}.$

By Buchberger's criterion the conclusion follows.\end{proof}

\smallskip
We prove now that  if $I^*_2 $ is square-free then the Hilbert function is strictly increasing,  so that the associated graded ring is  Cohen-Macaulay.

\smallskip
\begin{proposition}
\label{xyxz}
Let $A =R/I  $ be a   local ring  of type $(2,2).$  If $I^*_{\ 2} $ is square-free, then the Hilbert function of $A$ is strictly increasing  and thus
the associated graded ring $gr_{\m}(A)$ is Cohen-Macaulay.
 \end{proposition}
\begin{proof}
 From Proposition \ref{preslemma} $(ii) $  we may  assume,  up to isomorphism of $R$, that  $$
I=(x^2 +  x z + F(y,z), x y +  b y z+ \alpha y^r + \beta z^s)
$$
where
$F\in \seryz_{\ge 3}, $
$b\in \seryz$ with $b(0,0)=1$,
$r, s\ge 3$,
$\alpha=0$ or $\alpha \in \mathbb U(\sery)$, and
$\beta=0$ or $\beta \in \mathbb U(\serz)$.

If $HF_{R/I}(n+2)=n+2, $ then $e=n+2 $ and we conclude by Proposition \ref{increasing1}.
Assume that $HF_{R/I}(n+2)=n+3$, then by \lemref{preparation} we have that $x z^n\in\lt{I}$.
By Burchberger's criterion we should have
$xz^n =\lt{\NF(S(f,g),\{f,g\})}$.
The $S$-polynomial of the pair $f,g$ is
$$
h:=S(f,g)= z(b-1) A + \alpha y^{r-1} A + y F -\beta x z^s,
$$
$A= b yz + \alpha y^r + \beta z^s$. We write
$L=z(b-1) A + \alpha y^{r-1} A + y F $; notice that $L\in \seryz$ and $\beta\in\serz$
so $h=\NF(h, \{f,g\})$.
Since $\lt{h}= xz^n$ we deduce
$\beta \in \mathbb U(\serz)$, $s=n$, and  $\order(L)\ge n+1$.

\noindent
Let now consider
\begin{eqnarray*}
S(h,g)=  W &=& \beta z^n g + y h \\
            &=& \beta b y z^{n+1}+ \alpha \beta y^r z^n + \beta^2 z^{2n} +y L.
\end{eqnarray*}
\noindent
Notice that since $b, \beta \neq 0$
$$
\order(\alpha \beta y^r z^n),  \order(\beta^2 z^{2n})\ge n+3>n+2=\order(\beta b y z^{n+1})
$$
and $\lt{\beta b y z^{n+1}}=y z^{n+1}$.

Recall that $\order(L)\ge n+1$, so in order to prove that  $\lt{W}=y z^{n+1}$ we should prove  that in $\supp(y L)$
there is not the monomial $y z^{n+1}$.
This is equivalent to prove that in  $\supp( L)$
there is not the monomial $z^{n+1}. $ At this end   we set $y=0$ in $L$ and we get
$$
L(0,z)= (b(0,z)-1)\beta  z^{n+1}.
$$
recall that $b(0,0)=1$ so $\order(L(0,z))\ge n+2$.
Hence we have that $\lt{k}=y z^{n+1}$.

Let us consider now the monomial ideal  $J=(x^2, xy, x z^n, y z^{n+1})\subset \lt{I}$.
We have
$$
HF_{R/I}(n+2) \le HF_{R/J}(n+2) = n+2,
$$
a contradiction.
 \end{proof}

\vskip 2mm
Notice that if $I^*_2$ is square-free then by Lemma \ref{initial} $(ii)$ we may assume, up to isomorphisms, that $f^*=x y, g^*=xz$.
Hence Proposition \ref{xyxz} recover \cite[Corollary 4.6]{GHK07} in the case $(2,2)$.

 \vskip 2mm
 The following example shows that the converse of the above theorem does not hold.
 Let $I=(x^2-y^2z,xy-y^3)\subseteq \ser3 $. It is clear that $x^2\in I^*_2$ so that  $I^*_2$ is not square-free. It is easy to see  that the Hilbert function of $A=\ser3 /I$ is strictly increasing, namely is $\{1,3,4,5,5,5,5,.......\}.$ By Proposition \ref{increasing1} the  associated graded ring of $A$ is Cohen-Macaulay.

 We close this section by describing the possible minimal free resolutions of the associated graded ring of a local ring of type $(2,2).$

 We have seen in (\ref{flat}) that the Hilbert function of a local ring $A$  of type $(2,2)$ has the following shape

\begin{equation}  HF_A(t)=
\begin{cases}
     1 & \ \ \text{$t=0$}, \\
     t+2 & \ \ \text{$t=1,\cdots, n$},\\
     t+1 & \ \ \text{$t=n+1,\cdots, e-1$},\\
     e & \ \ \text{$t\ge e$}.\\
\end{cases}
\end{equation} where  $n$ is the least integer such that $HF_A(n)=HF_A(n+1).$ We have $3\le n\le e-2$ and it is easy to see that the lex-segment ideal with the above Hilbert function is the following ideal $L:=(x^2,xy,xz^n,y^e).$
We can compute the  minimal free resolution of $P/L$ by using the well known formula of Eliaouh and Kervaire. We get

$$0\to P(-n-3)\to P(-3)\oplus P(-n-2)^2\oplus P(-e-1)\to $$ $$\to  P(-2)^2\oplus P(-n-1)\oplus P(-e)\to P\to P/L \to 0.$$

 It is clear that in the case $e\ge n+3$ there is no possible cancelation. Hence every homogeneous ideal $J$ with Hilbert function as in (\ref{flat}) and with $e\ge n+3$ has the same resolution of the corresponding lex-segment ideal.

In the other case, when $e=n+2,$ we can either have the above resolution  or one of the following obtained by cancelation:
$$0\to P(-n-3)\to P(-3)\oplus P(-n-2)\oplus P(-n-3)\to P(-2)^2\oplus P(-n-1)\to P\to P/J\to 0$$

$$0\to P(-3)\oplus P(-n-2)\to P(-2)^2\oplus P(-n-1)\to P\to P/J \to 0.$$ It is clear that if $P/J$ is Cohen-Macaulay only the last shorter resolution is available.

We apply this to the associated graded ring of a local ring of type (2,2) and  we get the following result.

\smallskip
\begin{proposition}
\label{resolution} Let $A$ be a local ring of type (2,2),    $e$ the multiplicity of $A$ and let $n$ be an integer such that $n\le e-2.$  If  $e=n+2, $ then  $gr_{\m}(A) $ is Cohen-Macaulay with minimal free resolution:
$$0\to P(-3)\oplus P(-n-2)\to P(-2)^2\oplus P(-n-1)\to P \to gr_{\m}(A)  \to 0.$$ If  $e\ge n+3, $ then $e\le 2n$ and $gr_{\m}(A) $ is not Cohen-Macaulay with minimal free resolution
$$0\to P(-n-3)\to P(-3)\oplus P(-n-2)^2\oplus P(-e-1)\to $$ $$\to  P(-2)^2\oplus P(-n-1)\oplus P(-e)\to P\to gr_{\m}(A)  \to 0.$$
\end{proposition}
\begin{proof} It is enough to remark that by Proposition \ref{increasing1} the associated graded ring of a local ring of type (2,2) is Cohen-Macaulay when $e=n+2.$
\end{proof}

\smallskip
\section{A structure's theorem  for quadratic complete intersections of codimension two}

The aim of this section is to give a structure, up analytic isomorphism, of the minimal system of generators of ideals $I$ of type $(2,2) $  such that $A=\ser3/I$ is of multiplicity $e$.
This a first step towards the difficult problem of the analytic classification of the ideals of type $(2,2)$.
In this direction we show in \exref{noniso} two ideals of type $(2,2)$ with same Hilbert function that are not analytic isomorphic.
 Accordingly with \propref{base1}, \exref{ex}, and \exref{examples2n} the Hilbert function of $A$ take exactly the following shapes
 \begin{equation}   H(e)(t) :=
\begin{cases}
     1 & \ \ \text{$t=0$}, \\
     t+2 & \ \ \text{$t=1,\cdots, e-3$},\\
     e  & \ \ \text{$t\ge e-2$}
\end{cases}
\end{equation}
or
\begin{equation}  H(n,e) (t):=
\begin{cases}
     1 & \ \ \text{$t=0$}, \\
     t+2 & \ \ \text{$t=1,\cdots, n$},\\
     t+1 & \ \ \text{$t=n+1,\cdots, e-2$},\\
     e & \ \ \text{$t\ge e-1$}.\\
\end{cases}
\end{equation}

\smallskip
\begin{theorem}
\label{conjecture22}
Let $A $ be a local ring  of type $(2,2)$ and multiplicity $e.$
The following conditions are equivalent:

\noindent
$(i)$  $HF_{A} = H(n,e) $ for some integer $n \ge 3.$

\noindent
$(ii)$ Up analytic isomorphism, $I$ is generated in $R=\ser3 $ by:
\begin{eqnarray*}
  f &=& x^2 + a  z^p (x+H)  -H^2 +L  \\
 g &=& xy + \alpha z^n + y  H
\end{eqnarray*}
where \begin{itemize}
\item  $a\in \{0,1\}$, $p\ge 2$,   $\alpha\in \mathbb U (\serz)$,

\item $H, L  \in \seryz$ with $\order(L)\ge n+1$ , $\order(H)\ge 2$,

\item $n+3 \le e \le 2n,  $

\item $ \order (2\alpha z^nH -a \alpha z^{n+p} +y L)  \ge   e-1  $ and the equality holds whenever   $e< 2n.  $
\end{itemize}
\end{theorem}
\begin{proof}
Taking advantage of  Proposition \ref{preslemma}, the proof is based on the computation of $\lt I$ accordingly with Buchberger's criterion.   As usual assume $x> y, x> z.$

\noindent First we prove $(i)$ implies $(ii)$.
Since $HF_{A} = H(n,e),  $ then  by Proposition \ref{increasing1} $gr_{\m}(A) $ is not Cohen-Macaulay and, by Theorem \ref{theorem1}, $n+3 \le e \le 2n. $ By Proposition \ref{xyxz}, $I^*_2 $  contains a square of a linear form. Hence  we may assume that
$f^*=x^2$ and $g^*=xy$. Notice that $x^2, xy \in \lt I, $ hence because $(i), $  by Lemma \ref{preparation},   $\lt I \supseteq  (x^2, xy, xz^n). $ From the particular shape of the Hilbert function it is easy to see that $\lt I= (x^2, xy, xz^n, m)$ where $m $ is a monomial in $ K[y,z]_e.$
From \lemref{preslemma} we may  also assume
\begin{eqnarray*}
  f &=& x^2 + a x z^p + F(y,z),\\
g &=&  xy + G(y,z),
\end{eqnarray*}
where $a\in \{0,1\}$, $p\ge 2$,  $F, G \in \seryz$ with $\order(F), \order(G)\ge 3$.
Moreover, from the equation \eqref{g} of the  proof of Theorem \ref{theorem1},   (\ref{g}),  we get
$$G(y,z)= y H(y,z)+\alpha z^n $$ where    $H\in \seryz_{\ge 2}$ and
  $\alpha \in \mathbb U(\serz).$
We recall that   $S(f,g)= yf-xg = a xy z^p +y F - \alpha xz^n -xy H. $ In particular a standard computation gives
  $$ h:= \NF(S(f,g),\{f,g\}) = - \alpha x z^n  + yL + \alpha z^n (H- a  z^p )$$ where
  $L= F- a z^p H +H^2.$  Notice that $\order (\alpha z^n (H- a  z^p ) ) \ge n+2.  $
  Notice that $xz^n =\lt h$.

  A simple calculation shows that  $\NF(S(h,f),\{h,f,g\})=0$.
  On the other hand
  $$S(h, g) = \NF(S(h,g),\{h,f, g\}) = \alpha^2 z^{2n} +y (2\alpha z^nH -a \alpha z^{n+p} +y L) \neq 0 $$
 because $  \alpha \neq 0$ and $z^{2n} $ does not appear in the support of the remaining part.
  As a consequence $m=\lt {S(h, g)}, $ and hence $ \order (S(h,g))=e.$ It follows $ \order (2\alpha z^nH -a \alpha z^{n+p} +y L)      \ge e-1$.
  In particular  $\order (L) \ge n+1$, and $ \order (2\alpha z^nH -a \alpha z^{n+p} +y L)=e-1$ if $e < 2n$,.

 Conversely,  assuming  $(ii), $ it is enough to apply Buchberger's criterion for computing $\lt I.$   By following the previous computations we get $\lt I=(x^2, xy, xz^n, m) $ where $m=\lt{\alpha^2 z^{2n} +y (2\alpha z^nH -a \alpha z^{n+p} +y L)} $ and $(i) $ follows.
  \end{proof}

  \vskip 2mm
  If the Hilbert function is increasing, i.e. of type $H(e)$, we present a structure's theorem under the assumption that $I^*$   does not contain the square of a linear form.

 \smallskip
\begin{theorem}
\label{conjecture23}
Let $A $ be a local ring  of type $(2,2)$ and multiplicity $e.$
The following conditions are equivalent:

\noindent
$(i)$  $HF_{A} = H(e) $ and $I^* $ does not contain the square of a linear form

\noindent
$(ii)$ Up analytic isomorphism, $I$ is generated in $R=\ser3 $ by:
\begin{eqnarray*}
  f &=& x^2 +   x z +F \\
  g &=& x y +  d y z+ \alpha y^r + \beta z^s
\end{eqnarray*}
where \begin{itemize}

\item $r \ge  3$,

\item $F\in \seryz$ and $\order(F)\ge 3$,

\item  $d\in  \mathbb U(\seryz)$, with $d(0,0)=1$,

\item $\alpha=0$ or $\alpha \in \mathbb U(\sery)$,
$\beta=0$ or $\beta \in \mathbb U(\serz)$ and $s \ge e-1\ge 3,  $

\item
$\order(F +  d(d-1) z^2 + \alpha (2d-1) z  y^{r-1} + \alpha^2 y^{2(r-1)}) =  e-2.
$
\end{itemize}
\end{theorem}
\begin{proof}
As usual, consider a monomial ordering $\tau $ with $x>y, x>z.$
We prove  $(i) $ implies $(ii).$ By  Theorem \ref{preslemma} $(ii)$, we may assume  that
$$
I=(x^2 +   x z + F(y,z), x y +  d y z+ \alpha y^r + \beta z^s)
$$
$F\in \seryz_{\ge 3}, $
$d\in \seryz$ with $d(0,0)=1$,
$r, s\ge 3$,
$\alpha=0$ or $\alpha \in \mathbb U(\sery)$, and
$\beta=0$ or $\beta \in \mathbb U(\serz)$.
 Since the Hilbert function is increasing up to $n=e-2\ge 2 $ and $HF_{R/I}(t) = e$ for all $t\ge e-2, $ then
  $\lt{I} =(x^2, xy, m ) $ where $m \in K[y,z] $ is a monomial of  degree
$e-1$.

By Buchberger's criterion necessarily
$m= \lt{\NF(S(f,g),\{f,g\})}$.
Now
$$
S(f,g)= -\beta xz^s  + yF(y,z) +xy [(1-d)z - \alpha y^{r-1}]
$$
After a computation we get
$$\NF(S(f,g),\{f,g\}) = -\beta xz^s  +  yW +\alpha \beta y^{r-1} z^s + \beta (d-1) z^{s+1} $$
where   $W=F +  d(d-1) z^2 + \alpha (2d-1) z  y^{r-1} + \alpha^2 y^{2(r-1)}.$
Since $y W \in \seryz$, $r\ge 3$ and $1-d\in (y,z)\seryz$ we get that if $\beta \neq 0, $ then $  xz^s $ appears in the support of $\NF(S(f,g),\{f,g\}).$ Since $\lt{\NF(S(f,g),\{f,g\})} \in K[y,z]_{e-1},$  it follows  $\order(W)=  e-2 $ and, if $\beta \neq 0, $ then $s\ge e-1. $

Conversely if we assume $(ii),$ then it is easy to see that $I^* $ does not contain the square of a linear form because $I^*_2= (x^2+xz, xy+yz) $ which is reduced. Moreover
 by repeating  Buchberger's algorithm, looking at the previous computation on  $S(f,g), $  we  get  $$\lt{I}=(x^2, xy, y \lt W),  $$ hence  $HF_{A} = H(e). $  \end{proof}

\smallskip
\section{Examples }
The aim of  this section is to  present   examples  supporting   the results of the previous sections or detecting  the possible extensions  to the non quadratic case. All  computations are performed by using CoCoA system (\cite{CoCoA}).
Here $HS_A(\theta) $ denotes the Hilbert series of $A, $ that is $HS_A(\theta)= \sum_{t\ge 0} HF_A(t)\theta^t.$
\vskip 5mm   We have seen in Proposition \ref{resolution} that the minimal free resolution of the tangent cone of  a local ring of type $(2,2)$ has no possible cancelation, both in the case the Hilbert function is strictly increasing and in the case of a flat.  One can ask if this is the case also  for local rings of type $(a,b)$ with $3\le a\le b.$

The first two examples that we propose show that the answer is negative.

\begin{example} Let $A=R/I$   where  $I=(x^3, z^5 + xz^3 + x^2y).$ The local ring $A$ has type $(3,3)$ and $I^*=(x^3, x^2y, x^2z^3, -xyz^5 + xz^6, -xz^7, z^{10}).$
The resolution of $P/I^*$ is the following

$$0\to P(-7)\oplus P(-10) \to  P(-4)\oplus P^2(-6)\oplus P(-8)\oplus P^2(-9)\oplus P(-11) \to$$ $$\to  P^2(-3)\oplus P(-5)\oplus P(-7)\oplus P(-8)\oplus P(-10)\to P \to P/I^*\to 0.$$ It is clear that we have a possible cancelation and the Hilbert function  $$\{1,3,6,8,10,11,13,14,14,15,15,.......\}$$ has a flat in position 7.
\end{example}

\smallskip
\begin{example} Let $A=R/I$ where $I=(x^4, z^4 + x^2y).$ The local ring $A$ has type $(3,4)$ and $I^*=(x^2y, x^4, x^2z^4, z^8).$
The resolution of $P/I^*$ is the following
$$0 \to P(-9) \to P(-5)\oplus P(-7)\oplus P(-8)\oplus P(-10)\to $$ $$ \to P(-3)\oplus P(-4)\oplus P(-6)\oplus P(-8)\to P \to P/I^*\to 0.$$ It is clear that we have a possible cancelation and the Hilbert function $$\{1,3,6,9,11,13,14,15,16,16,16,........\}$$ is strictly increasing.
\end{example}

\smallskip
 The following example   shows   that Proposition \ref{increasing1}  cannot be extended to local rings of type $(a,b) $ with $a>2.$

\begin{example}   Let us consider the ideal $ I=(x^4, x^2y+z^4) \subseteq R=k[[x,y,z]].  $  Then $A=R/I$ has  strictly increasing Hilbert function, in fact the Hilbert series is:
$$HS_A(\theta) = (1 + 2\theta + 3\theta^2 + 3\theta^3 + 2\theta^4 + 2\theta^5 + \theta^6 + \theta^7 + \theta^8) / (1-\theta).$$
Nevertheless $I^*= (x^2y, x^4, x^2z^4, z^8), $ hence $gr_{\m}(A) $ is not Cohen-Macaulay.
\end{example}

\smallskip
The following example, due to T. Shibuta,  shows that    the Hilbert function of a one-dimensional  local domain   of type $(2,b) $ can  have  $b-1$ flats, the maximum number accordingly to Proposition  \ref{increasing}.

\begin{example} (see \cite{GHK07}, example 5.5) Let $b \ge 2 $ be an integer. Consider the family of semigroup rings
$$A= k[[ t^{3b}, t^{3b+1}, t^{6b+3}]].$$
It is easy to see that   $A= k[[x,y,z]]/I_b $ where $I_b= (xz-y^3, z^b-x^{2b+1}). $ Thus $A$ is a one-dimensional local domain of type $(2,b). $ For every $b \ge 2 $ the Hilbert function of $A$ has    $b-1$ flats. Namely
\begin{equation} \label{b-1flats} HF_A(t)=
\begin{cases}
     1 & \ \ \text{$t=0$}, \\
    2 t+2 & \ \ \text{$t=1,\cdots, b-1$},\\
    2 b & \ \ \text{$t=b$},\\
    2b+1 & \ \ \text{$t=b+1$},\\
    2b+k  & \ \ \text{$t=b+2k, \ \ k=1,\cdots, b-1$},\\
    2b+k+1  & \ \ \text{$t=b+2k+1, \ \ k=1,\cdots, b-1$},\\
    3b  & \ \ \text{$t\ge 3b-1 $}.\\
\end{cases}
\end{equation}
\end{example}

\smallskip
In the above  example the Hilbert function of the local ring of type $(2,b) $ presents  $b-1 $ flats which are not consecutive. The following example shows  that we can also have $b-1$ consecutive flats, that is a  strip like this:  $HF(n)=HF(n+1)= \dots =HF(n+b-1) <e.$

\begin{example}   Let us consider  the ideal $ I=(x^2,   xy^2+z^5+xy^3z^2) \subseteq R=k[[x,y,z]].  $  Then $A=R/I$  is of type $(2,3)$ and its  Hilbert function   presents two (=b-1) flats   which are consecutive: namely we have $HF(5)=HF(6)=HF(7) = 8 < e=10.$ In particular the Hilbert  series is:
$$HS_A(\theta)=(1+2\theta+2\theta^2+\theta^3+\theta^4+\theta^5+\theta^8+\theta^9)/(1-\theta).$$
\end{example}

\smallskip
 The Hilbert function of a local ring of type $(a,b), $ with $3 \le a\le b, $    is at the moment  far from our understanding.  In order to show how the problem is difficult when $a$ and $b$ are increasing, we present two more examples, the first of type $(3,3) $ with one very large platform consisting of 13  consecutive flats, the second of type $(4,4) $ with nine flats and three platforms.

\begin{example}   Let  $ I=(x^3-zy^{14},  x^2y+xz^7) \subseteq R=k[[x,y,z]].  $  The local ring $A=R/I$  is of type $(3,3)$ and
$$HF_A(15)=HF_A(16)=\dots \dots = HF_A(29)=31 < e=32. $$ In particular the Hilbert  series is:
$$ HS_A(\theta)=(1 + 2\theta + 3\theta^2 + 2\theta^3 + 2\theta^4 + 2\theta^5 + 2\theta^6 + 2\theta^7 + 2\theta^8 + \theta^9 + 2x\theta^{10} + 2\theta^{11} + 2\theta^{12} + $$ $$+ 2\theta^{13} + 2\theta^{14}  + \theta^{15} + \theta^{16} + \theta^{30}+ \theta^{31}) / (1-\theta)  $$
and $I^*=(x^3, x^2y, x^2z^7, xz^{14}, xy^{15}z, y^{31}z).$
\end{example}

\begin{example}   Let  $ I=(x^4,   xy^3 - z^6) \subseteq R=k[[x,y,z]].  $  The local ring $A=R/I$  is of type $(4,4) $ and
$$HF_A(8)=HF_A(9)=HF_A(10)=HF_A(11) =18; $$  $$HF_A(13)=HF_A(14)=HF_A(15)=HF_A(16) =20; $$ $$HF_A(18)=HF_A(19)=HF_A(20)=HF_A(21) =22< e=24; $$ In particular the Hilbert  series is:
$$HS_A(\theta)= (1 + 2\theta + 3\theta^2 + 4\theta^3 + 3\theta^4 + 2\theta^5 + \theta^6 + \theta^7 + \theta^8 + \theta^{12} + \theta^{13 }+ \theta^{17}+ \theta^{18} + \theta^{22} + \theta^{23}) / (1-\theta) $$ and $I^*= (xy^3, x^4, x^3z^6, x^2z^{12}, xz^{18}, z^{24}). $
\end{example}

\smallskip
It would be very interesting to describe the  isomorphism classes of local rings of type $(2,2)$  which have the same given Hilbert function. But this is a difficult task, as  the following examples show.

First we are given  the Hilbert function  $\{1,3,4,5,5,6,6,...\}$ which has a flat in position 3  and multiplicity 6.  The two ideals which we are going to prove that are not isomorphic are obtained one from the other with very little modifications, namely by adding a monomial to one of the two generators.

\begin{example}
\label{noniso}
Let us consider the ideals $$I=( x^2-y^4,xy+z^3), \ \ J=(x^2+xz^2-y^4, xy+z^3)$$
in $R=\ser3.$

They are of type $(2,2)$,
  they have  the same Hilbert function $\{1,3,4,5,5,6,6,...\}$ and the same leading ideal
$\lt{I}=\lt{J}=(x^2, xy, xz^3, y^6).$
On the other hand the ideals of initial forms differ in degree $6$:
$$I^*=(x^2, xy, xz^3, y^6-z^6),\ \ \ J^*=(x^2, xy, xz^3, y^6+yz^5-z^6).$$ We prove that $\ser3/I$ and $\ser3/J$ are not isomorphic.

If there exists an analytic isomorphism $\phi$ such that $\phi(I)=J$ then we can find power series $f,g,h$ of order $1$ such that $\mathcal M=(f,g,h)$ and $\phi$ is the result of substituting  $f$ for $x$, $g$ for $y$ and $h$ for $z$ in any power series of $R.$ We have  $f=L_1+F\ \ \ \  g=L_1+G\ \ \ \ h=L_3+H$ where $L_1,L_2,L_3$ are linearly independent linear forms  in $K[x,y,z]$ and $F,G,H$ are power series of order $\ge 2.$

We let for $i=1,2,3$
$$L_i=\lambda_{i 1}x+\lambda_{i 2}y+\lambda_{i 3}z$$ with $\lambda_{i j}\in \res.$
Since $x^2-y^4\in I$ we have $\phi(x^2-y^4)=f^2-g^4\in J$, hence $L_1^2\in J^*.$ Since $I^*_{\ 2}$ is the $\res$-vector space $I^*_{\ 2}=<x^2,xy>,$  we have
$$(\lambda_{1 1}x+\lambda_{1 2}y+\lambda_{1 3}z)^2=px^2+qxy$$  with $p,q\in \res;$ this clearly implies $\lambda_{12}=\lambda_{13}=0.$

In the same way, since  $xy+z^3\in I,$ we have $\phi(xy+z^3)=fg+h^3\in J$, hence $L_1L_2\in J^*.$ Thus we get $$(\lambda_{1 1}x)(\lambda_{2 1}x+\lambda_{2 2}y+\lambda_{2 3}z)=rx^2+sxy$$  with $r,s\in \res.$ This implies   $\lambda_{23}=0$ because $\lambda_{1 1}\neq 0.$

Finally we have $$y^6-z^6=-y^2(x^2-y^4)+(xy+z^3)(xy-z^3)\in I$$ so that $\phi(y^6-z^6)=g^6-h^6\in J,$ and,  as before, $L_2^6-L_3^6\in J^*.$  Looking at the generators of the vector space $J^*_{\ 6}$ we get as a consequence
$$(\lambda_{2 1}x+\lambda_{2 2}y)^6-(\lambda_{3 1}x+\lambda_{3 2}y+\lambda_{3 3}z)^6=Ax^2+Bxy+Cxz^3+D(y^6+yz^5-z^6)$$ where $A, B, C, D$ are forms of degree $4, 4, 2, 0$ respectively in the polynomial ring $\res [x,y,z].$

Since $L_1, L_2, L_3$ are linearly independent, we must have $\lambda_{3 3}\neq 0.$ Hence,  looking at the coefficient of the monomial $y^5z$ in the above formula, we get $\lambda_{3 2}=0.$ But then, looking at the coefficient of the monomial $yz^5,$ we certainly  get $D=0$ and finally, looking at the coefficient of the monomial $z^6,$ we get $\lambda_{3 3}=0.$ This  is a contradiction, so that the algebras $R/I$ and $R/J$ are not in the same isomorphism class.
\end{example}

\smallskip
The case when the Hilbert function is strictly increasing is not more easy to handle. Here we consider the Hilbert function $\{1,3,4,5,6,6,6,....\}$ which is strictly increasing and we look at the possible isomorphism classes of local rings with that Hilbert function.

\begin{example}
Let us consider  the two ideals $$I:=(x^2+y^4,xy),\ \ \ \ \ J:=(x^2+y^4+z^4,xy).$$
They have the same Hilbert function  $\{1,3,4,5,6,6,6,....\}$ and different tangent cones, namely $$I^*=(x^2,xy,y^5)\ \ \ \ \ \  J^*=(x^2,xy,y^5+yz^4).$$ A calculation as before shows that $\ser3/I$ and $\ser3/J$  are not isomorphic.
\end{example}

\providecommand{\bysame}{\leavevmode\hbox to3em{\hrulefill}\thinspace}
\providecommand{\MR}{\relax\ifhmode\unskip\space\fi MR }
\providecommand{\MRhref}[2]{%
  \href{http://www.ams.org/mathscinet-getitem?mr=#1}{#2}
}
\providecommand{\href}[2]{#2}

\smallskip
\noindent
Juan Elias\\
Departament d'\`Algebra i Geometria\\
Universitat de Barcelona\\
Gran Via 585, 08007 Barcelona, Spain\\
e-mail: {\tt elias@ub.edu}

\smallskip
\noindent
Maria Evelina Rossi\\
Dipartimento di Matematica\\
Universit{\`a} di Genova\\
Via Dodecaneso 35, 16146 Genova, Italy\\
e-mail: {\tt rossim@dima.unige.it}

\smallskip
\noindent
Giuseppe Valla\\
Dipartimento di Matematica\\
Universit{\`a} di Genova\\
Via Dodecaneso 35, 16146 Genova, Italy\\
e-mail: {\tt valla@dima.unige.it}

\end{document}